\documentclass[10pt]{amsart}
\usepackage{amsmath,amssymb,amsthm}

\usepackage{float}
\usepackage[all,cmtip]{xy}
\usepackage{tikz}
\usetikzlibrary{matrix}
\usetikzlibrary{cd}
\usepackage{subfig}
\usepackage[colorlinks,hyperindex,linkcolor=blue]{hyperref}
\hypersetup{
 pdfauthor = {Javier Martínez-Aguinaga and Sergio de María},
 pdftitle= {Isotopies between projectivizations of knots: a diagrammatic algorithm},
 pdfsubject = {Knot Theory, Geometric Topology, Algebraic Topology}}

\newcommand{\BB}{\mathbb{B}}

\newcommand{\PP}{\mathbb{P}}
\newcommand{\RR}{\mathbb{R}}

\newcommand{\ZZ}{\mathbb{Z}}

\newcommand{\NS}{\mathbb{S}}

\newcommand{\SO}{\textup{SO}}

\newtheorem{proposition}{Proposition}[section] 
\newtheorem{theorem}[proposition]{Theorem}
\newtheorem{definition}[proposition]{Definition}

\newtheorem{remark}[proposition]{Remark}

\newtheorem{example}[proposition]{Example}

\makeatletter
\newcommand{\superimpose}[2]{%
  {\ooalign{$#1\@firstoftwo#2$\cr\hfil$#1\@secondoftwo#2$\hfil\cr}}}
\makeatother

\title{A constructive solution to the equivalence problem for knot projectivizations}


\author{Sergio de María}
\address{Universidad Complutense de Madrid, Departamento de Álgebra, Geometría y Topología, Facultad de Ciencias
Matemáticas. 28040, Madrid, Spain.}
\email{sergiodm@ucm.es}

\author{Javier Mart\'{i}nez-Aguinaga}
\address{Universidad Complutense de Madrid, Departamento de Álgebra, Geometría y Topología, Facultad de Ciencias
Matemáticas. 28040, Madrid, Spain.}
\email{frmart02@ucm.es}

\emergencystretch=\maxdimen
\begin{document}

\begin{abstract} 

The problem of whether different projectivizations of the same affine knot $K\subset\NS^3$ are equivalent in $\RR\PP^3$ can be found in \cite{Kauffman} and has also been posed as an open question in \cite{Narayanan2025}. In this note we provide a constructive solution to the problem. In particular, we adapt an idea due to A. Hatcher developed in the realm of embedding spaces and we describe an algorithm that produces an explicit isotopy between any two given projectivizations of the same affine knot. More generally, we introduce the notion of \textit{lensification} of a knot in any lens space $L(p,q)$ and describe an algorithm that works in that more general setting, of which $\RR\PP^3\simeq L(2,1)$ is a particular instance. Finally, we apply this algorithm to several pairs of knots from the literature for which the equivalence problem was raised as an open question, finding explicit isotopies.

\end{abstract}
\maketitle
\addtocontents{toc}{\setcounter{tocdepth}{1}}

\section{Introduction}\label{intro}

 There has been significant development in the Theory of Knots in the real projective space $\RR\PP^3$ in recent years; see \cite{viro, Mishra2023, ManolescuWillis, Kauffman, Narayanan2025, CattabrigaManfredi2018, manfredi2014knots, CATTABRIGA2013430}. In a recent article \cite{Kauffman}, L. H. Kauffman, R. Mishra and V. Narayanan relate the study of knots in the real projective space $\RR\PP^3$ with the Theory of Virtual Knots, showing deep connections between the two. In particular, they define a correspondence that assigns a virtual knot to a given diagram of a projective knot or link. This allows to import invariants from Virtual Knot Theory to the projective setting. 
 
 They define a Jones type polynomial invariant for projective knots, which they then show to be equivalent to another polynomial invariant in Projective Knot Theory due to Y. V. Drobotukhina \cite{Drobotukhina1990}. They also import the Khovanov Homology from virtual knots in order to construct a Khovanov-type homology for projective knots and links. Likewise, by considering the Rasmussen invariant for virtual knots, they obtain an analogous invariant in the projective setting. These last two invariants motivated by the Theory of Virtual knots turn out to be essentially equivalent to invariants previously introduced in the work \cite{ManolescuWillis} of C. Manolescu and M. Willis by very different methods, as noted in \cite{Kauffman}. This showcases the deep relationship between Virtual Knot Theory and the Theory of Knots in $\RR\PP^3$.

 Nonetheless, although much progress has been made in the projective setting, there were certain knots that could still not be distinguished by any of the existing methods. In Section 7 \textit{``What this method cannot see!''} of \cite{Kauffman}, L. H. Kauffman, R. Mishra and V. Narayanan  noted  that all projectivizations of the same given affine knot cannot be distinguished by any of the aforementioned projective invariants, since the virtual knot assigned to all of them by the process mentioned above is exactly the same. They thus wonder  whether different projectivizations of the same given knot could represent the same knot in $\RR\PP^3$. In particular, they discuss this problem for the projectivization of the figure eight knot (see Subsection \ref{firstapplication}). Likewise, V. Narayanan \cite[Sec. 5]{Narayanan2025} makes the following comment in this regard referring to projectivizations of the same knot: ``\textit{The virtual
 knots associated with these would be the same. Hence none of the invariants described above will distinguish them. It is an interesting question to ask whether these are isotopic to each other}''.

We first observe that an affirmative answer to the question can be derived from a purely theoretical argument by reinterpreting the projectivization of an affine knot $K$ as a connected sum of a class-$1$ unknot in $\RR\PP^3$ with $K$ (see Remark \ref{rem:conSumLocal}). Nonetheless, the goal of this note is not only to answer the question but rather to provide a constructive and explicit solution to the problem. More specifically, we will describe such a solution in a purely diagrammatic way by working in disk diagrams (Subsection \ref{sec:diskdiagram}).

 For that purpose, we will make use of an idea of A. Hatcher developed in the realm of embedding spaces (in particular, described in order to represent the \textit{Fox-Hatcher} loop of long embeddings). Precisely, working in a disk diagram,  we will show how to diagrammatically represent an explicit isotopy between any two given projectivizations of the same affine knot. 


We will proceed in a more general fashion. We first introduce the notion of \textit{lensification} of a knot, which produces a non-local knot in any lens space $L(p,q)$ out of any given knot $K\subset\NS^3$. This notion naturally extends the notion of \textit{projectivization} of a knot, introduced by R. Mishra and V. Narayanan in \cite{Mishra2023}. We then describe how to adapt an idea by A. Hatcher to that more general setting so that we get an explicit isotopy between any two given lensifications in $L(p,q)$ of the same knot in diagrammatic terms. Then the case $\RR\PP^3\simeq L(2,1)$ just follows as a particular instance. As an application, in the last section we apply the algorithm to several examples from the literature \cite{Kauffman, Narayanan2025} for which the equivalence problem had been posed as a question and we depict explicit isotopies between them.

\textbf{Acknowledgements}: The second author would like to thank Ángel González-Prieto and Javier Martínez for useful discussions about knots in general $3$-manifolds as well as Enrique Aycart and Raquel Díaz for useful comments. The second author acknowledges support from PID2022-142024NB-I00 by MICINN (Spain).

 
\renewcommand{\thetheorem}{\thesection.\arabic{theorem}}

\section{Preliminaries}

In this section, we will recall some fundamental notions from the Theory of knots in lens spaces. We essentially follow the work \cite{manfredi2014knots} of E. Manfredi  and \cite{CATTABRIGA2013430} of A. Cattabriga, E. Manfredi and M. Mulazzani, from where we learnt about the topic.

\subsection{Knot theory in lens spaces}\label{subsect: modelo lentes}

Lens spaces $L(p,q)$ are compact, connected and orientable $3$-manifolds that can be defined in various ways. We will define these spaces through a geometric model described in \cite{manfredi2014knots} and \cite{CATTABRIGA2013430}. Henceforth, we will assume that $0\leq q <p$ are integers satisfying $\gcd(p,q)=1$.

    Consider  $B^3 :=\{x\in \RR^3 : \ |x|\leq 1\}$ the closed $3$-dimensional ball, or $3$-ball. We introduce the following terminology and notation following \cite{manfredi2014knots}: $E^+:=\{(x,y,z)\in \partial B^3 : z\geq 0\}$ and $E^-:=\{(x,y,z)\in \partial B^3 : z\leq 0\}$ denote the upper and lower hemispheres of $\partial B^3$, respectively. Sometimes, we will also call them northern and southern hemispheres, respectively. We write $N$ for the north pole of $B^3$ and $S$ for its south pole. Finally, we denote by $\mathring{B}=\{x\in\RR^3:|x|<1\}$ the interior of the $3$-ball $B^3$ and by $B_0^2$ the equatorial disk $B^3 \cap \{z = 0\}$ of the $3$-ball.

\begin{definition}[\textbf{Rotation map $g_{p,q}$}]
    We define the rotation map $g_{p,q}$ as the counterclockwise rotation of angle $\frac{2\pi q}{p}$ around the $z$ axis; i.e. for any point $x$ in the northern hemisphere $E^+$, we have:
    \begin{center}
    $\begin{array}{rccl}
	g_{p,q}: & E^+ & \longrightarrow & E^+ \\
	& x & \longmapsto &  A\cdot x,
	\end{array}$
    \end{center}
    where $$A = \begin{pmatrix}
    \cos \frac{2\pi q}{p} & -\sin \frac{2\pi q}{p} & 0 \\
    \sin \frac{2\pi q}{p} & \cos \frac{2\pi q}{p} & 0 \\
    0 & 0 & 1
    \end{pmatrix}$$ is the counterclockwise rotation matrix of angle $\frac{2\pi q}{p}$ around the $z$-axis.
    
\end{definition}

\begin{definition}[\textbf{Reflection map $f_z$}]
    We define the reflection map $f_z:E^+\to E^-$ as the map that maps any point $x$ in the northern hemisphere to its reflection with respect to the plane $\{z = 0\}$; i.e.
    \begin{center}
    $\begin{array}{rccl}
	f_z: & E^+ & \longrightarrow & E^- \\
	& (x,y,z) & \longmapsto & (x,y,-z).
	\end{array}$
    \end{center}
\end{definition}

\begin{figure}
    \centering
    \includegraphics[width=0.35\linewidth]{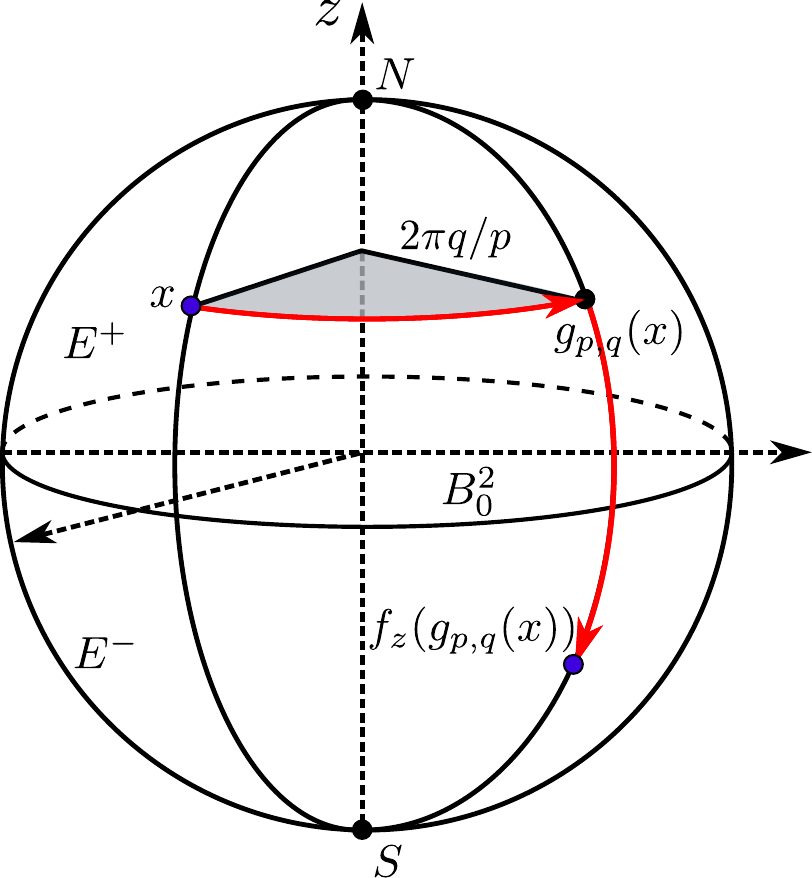}
    \caption{Visualization of the different elements in the $3$-ball model of $L(p,q)=B^3/\sim$. The northern and southern hemispheres are denoted by $E^+$ and $E^-$, respectively. The rotation map $g_{p,q}$ is depicted acting on the point $x$ and the reflection map $f_z$ is depicted acting on the point $g_{p,q}(x)$. Thus, the points depicted in blue $x\sim f_z\left(g_{p,q}(x)\right)$ are identified with each other.}
    \label{fig:espacio_lentes_3-bola}
\end{figure}

\begin{definition}\label{def:lensspace}   Let $p,q$ be two integers such that $\gcd(p,q) = 1$ and $0\leq q <p$. The \textbf{lens space}  $L(p,q)$ is defined as the quotient of the $3$-ball $B^3$ under the equivalence relation that identifies any point $x\in E^+$ with its image in $E^{-}$ by the map $f_z\circ g_{p,q}$; i.e. $x\sim f_z(g_{p,q}(x))$.
$$
    L(p,q) = B^3/_\sim.
    $$
    
We will denote by $F$ the quotient map $F: B^3 \to L(p,q)=B^3/\sim$. 
\end{definition}

\begin{remark} Note that the each equivalence class $|x|_\sim$ may contain a different number of points depending on where $x$ lies:
    \begin{itemize}
        \item[i)] If $x \in \mathring{B}^3$, then $|[x]_\sim| = 1$.
        \item[ii)] If $x \in \partial B^3 \setminus \partial B_0^2$, then $|[x]_\sim| = 2$.
        \item[iii)] If $x \in \partial B_0^2$, then $|[x]_\sim| = p$.
    \end{itemize}
\end{remark}

\begin{example}
The lens space $L(1,0)$ is equivalent to the $3$-sphere since there is no rotation taking place when defining the quotient. On the other hand, the lens space $L(2,1)$ is equivalent to the projective space $\RR\PP^3$. This is clear from its definition.
\end{example}
\subsubsection{Generator of $\pi_1(L(p,q))$ and $H_1(L(p,q))$ in the $3$-ball model.}\label{generatorH1}

Finally, let us describe the fundamental group of a lens space $L(p,q)$ in this $3$-ball model. Recall that $\pi_1(L(p,q))=\ZZ_p$ (we omit the basepoint from the notation since $L(p,q)$ is path-connected).

The generator of this group, whose class we denote by $[f]$, can be described geometrically as follows (see \cite{manfredi2014knots, CATTABRIGA2013430}). Any (counterclockwise oriented) curve $f$ given by a $p$-th fraction of the equatorial curve yields, at the homotopy level, the class $[f]$ which generates the fundamental group $\pi_1(L(p,q))$ (Figure \ref{fig: generatorH1}). Indeed, note that it represents a closed loop since its endpoints are identified points in $\partial B_0^2\subset B^3/\sim$. And, moreover, it has order $p$ (note that its $p$-times concatenation yields the whole equatorial curve, which is a contractible curve).

Since $L(p,q)$ is path-connected and has an abelian fundamental group, then the same generator $f$ yields the generator in homology; i.e. $[f]\in H_1(L(p,q))=\ZZ_p$ is the generator.

\begin{figure}[h!]
    \centering
    \includegraphics[width=0.3\linewidth]{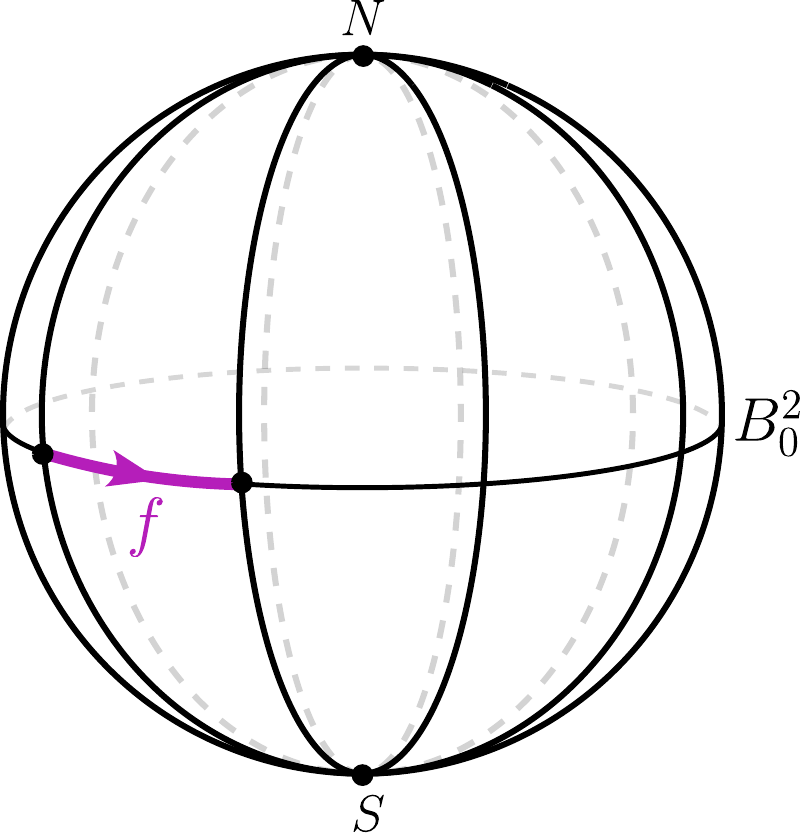}
    \caption{Depiction of the curve $f$ that yields the generator $[f]$ of $\pi_1(L(p,q))=\ZZ_p$ and also of $H_1(L(p,q))=\ZZ_p$. It is defined as a $p$-th fraction of the counterclockwise oriented equatorial curve. It depicts an actual closed loop since its endpoints are identified and it has order-$p$ since its $p$-times concatenation yields the whole equatorial curve that is contractible within the $3$-ball $B^3/\sim$.}
    \label{fig: generatorH1}
\end{figure}

\subsection{Disk diagrams}\label{sec:diskdiagram}

We will introduce the notion of a disk diagram (which we learnt from \cite{CATTABRIGA2013430} and \cite{manfredi2014knots}) of a knot or link in $L(p,q)$. This allows to study knots and links in lens spaces from a diagrammatic viewpoint. We follow the exposition from \cite{CATTABRIGA2013430} and \cite{manfredi2014knots}.

 Consider a link $L\subset L(p,q)$ and denote by $L'$ its preimage by the quotient map $F: B^3 \longrightarrow L(p,q)$; i.e.  $L' = F^{-1}(L)$. Henceforth, for the sake of clarity and whenever it is clear from the context, we will refer interchangeably as $L$ to both $L$ and $L'$.

\begin{proposition} \cite[p. 431]{CATTABRIGA2013430}. Every link $L\subset L(p,q)=B^3/\sim$ can be isotoped to another link $\tilde{L}$ such that:
    \begin{enumerate}
        \item[i)] $\tilde{L}$ does not pass through the poles $N,S$ of $B^3$ nor through the equator $\partial B_0^2$; i.e.
        \[
        \tilde{L}\cap \left(\{N,S\}\cup\partial B_0^2\right)=\emptyset.
        \]
        \item[ii)] The intersection of $\tilde{L}$ with $\partial B^3$ occurs transversely and in a finite (possibly empty) set of points.
    \end{enumerate}
\end{proposition}

\begin{definition}\label{obs: proyeccion_lenspace}
    Consider $x\in B^3$ and let $c_x$ be the circle (or line) that passes through $N, S$ and  $x$. We define the  \textbf{projection onto the equatorial disk} as the projection map
    
    \begin{center}
    $\begin{array}{rccl}
	p: & B^3 \setminus\{N,S\}& \longrightarrow & B_0^2 \\
	& x & \longmapsto & c_x\cap B_0^2.
	\end{array}$
    \end{center}
We will henceforth abuse notation and write $p$ as well for the map $p: \left(B^3 \setminus\{N,S\}\right)/\sim \,\to  B_0^2/\sim$ defined in the obvious manner in the quotient. By restricting the domain to the link $L$, we call the map
    $$
    p \big|_L : L \to B_0^2/\sim,
    $$
     the \textbf{projection of $L$ onto the equatorial disk}.
\end{definition}

    \begin{definition}
        Take a point $P$ in the image of the projection of a link $L$ onto the equatorial disk $p|_L$. We say that $P$ is a \textbf{multiple point} if the preimage of $P$ by $p|_L$, $\left(p|_L\right)^{-1}(P)$, contains more than one point. We say that $P$ is a \textbf{double point} if  $\left(p|_L\right)^{-1}(P)$ consists of exactly two points.  
    \end{definition}

\begin{definition}[\cite{CATTABRIGA2013430, manfredi2014knots}]\label{def:proyeccion_reg_lenspace}
    We say that the projection of a link $L\subset L(p,q)$,  $p \big|_L : L \to B_0^2$, is regular if 
    \begin{itemize}
        \item[i)] The projected curve $p|_L(L)$ does not contain cusps.
        
        \item[ii)] There are only a finite number of multiple points and they are all double points not lying in $\partial B_0^2$. Moreover, all self-intersections of $p|_L(L)$ at those points occur transversely.
    \end{itemize}
\end{definition}

We can define overpasses and underpasses for knots in lens spaces analogously to the classical case of $\RR^3$ or $\mathbb{S}^3$. We follow the reference \cite{CATTABRIGA2013430}.

\begin{definition}\cite[p. 432]{CATTABRIGA2013430}
Let $Q\in B_0^2$ be a double point of $K$ and consider its two preimages $\{q_1,q_2\}=p|_L^{-1}(Q)$. Assume that $q_1$ is closer to the north pole $N$ than $q_2$ and let $\alpha$ be an open subset within $L$ so that $p|_L(\bar{\alpha})$ does not meet $\partial B_0^2$ nor any other double point (here $\bar{\alpha}$ denotes the topological closure of $\alpha$). Then we say that $\alpha$ is an \textbf{underpass} relative to $Q$. We will use the term underpass to refer interchangeably to their projections, respectively.
\end{definition}

\begin{definition}\cite[p. 432]{CATTABRIGA2013430}
    Choose an underpass for each double point of $L$. Then we call \textbf{overpass} to each connected component of the complement of such underpasses within $L$. Analogously, we will use the term overpass to refer interchangeably to their projections.
\end{definition}    
    
\begin{definition}\cite[p. 432]{CATTABRIGA2013430}\label{def: disk_diagram} We call a \textbf{disk diagram} of a knot or link $L\subset L(p,q)=B^3/\sim$ to a regular projection of $L$ onto the equatorial disk $B_0^2$, where the corresponding overpasses and underpasses have been specified.
\end{definition}

\begin{remark}\label{obs: notacion_disk_diagram_lenspace}
    Additionally, for the sake of clarity, disk diagrams (Definition \ref{def: disk_diagram}) can be decorated with some extra information \cite{CATTABRIGA2013430}. Precisely, we can label the endpoints of the overpasses in the diagram by the following rule:
    \begin{enumerate}
        \item[i)] Following the counterclockwise orientation of the equator $\partial B_0^2$, label with  
        $+1, \cdots, +t$ the endpoints of the overpasses that lie in the upper hemisphere (see Figure \ref{fig: numeracion_disk_diagram}).
        \item[ii)] Now label with $-1,\cdots, -t$ the endpoints of the overpasses that lie in the lower hemisphere so that the point $-i$ is identified with the point $+i$ for every $i=1,\cdots, t$ (see Figure \ref{fig: numeracion_disk_diagram}).
    \end{enumerate}
    \end{remark}

\begin{figure}[h!]
    \centering
    \includegraphics[width=0.7\linewidth]{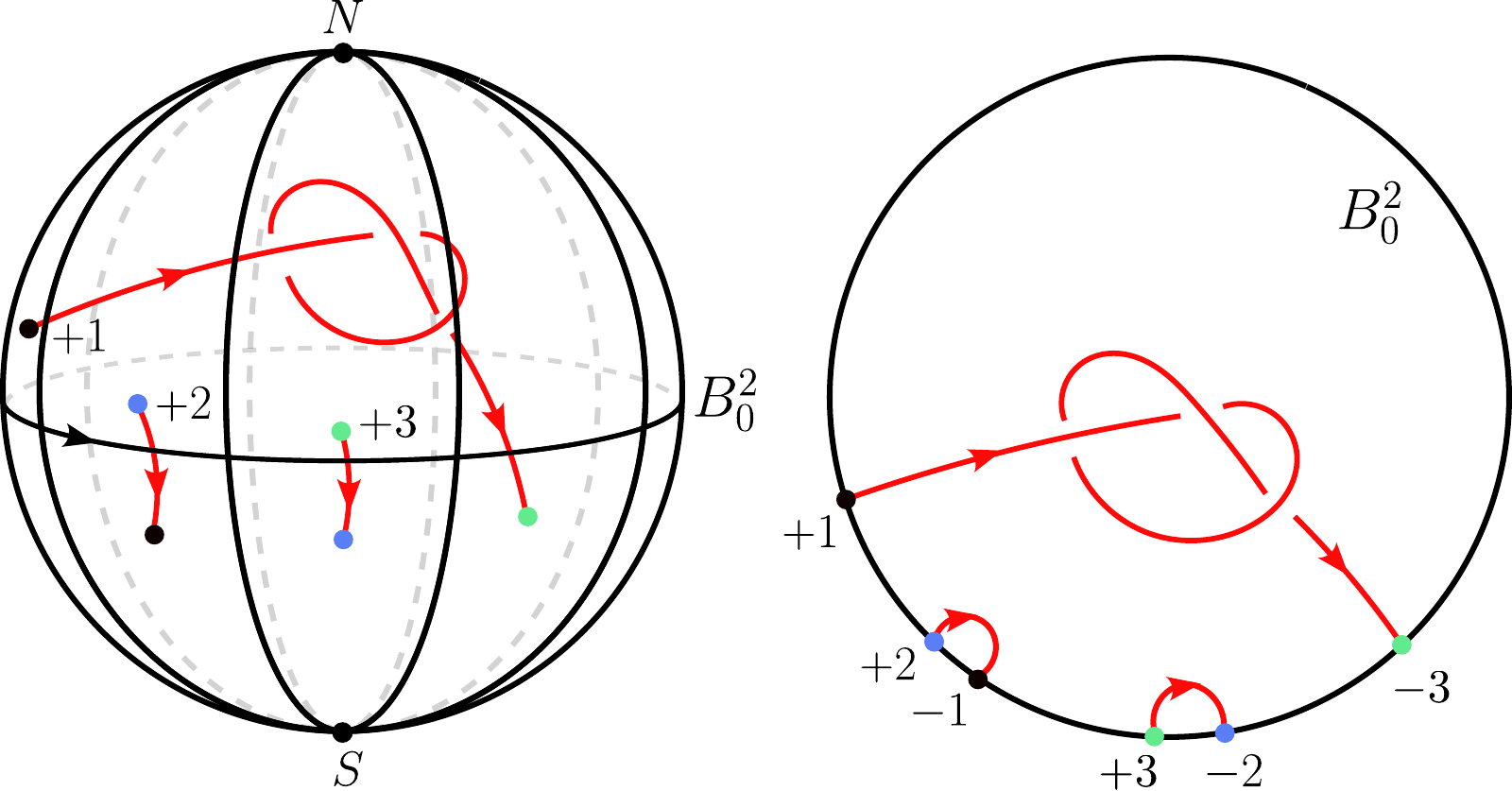}
    \caption{Depiction of a knot in a lens space $L(8,1)=B^3/\sim$ together with its associated disk diagram (Definition \ref{def: disk_diagram}). Points coming from the northern hemisphere are labelled with positive natural numbers, in order, following the counterclockwise orientation of the equator $\partial B_0^2$. Points coming from the southern hemisphere are labelled with negative numbers so that each point labelled with $-i$ is identified with its corresponding pair $+i$ and viceversa.}
    \label{fig: numeracion_disk_diagram}
\end{figure}

\section{Lensification of knots}\label{lensificado}

In this section we will introduce the notion of \textit{lensification} of a knot $K\subset \mathbb{S}^3$. This is an operation that takes a knot $K\subset \mathbb{S}^3$ as an input and produces a non-local knot $K_{\mathcal{L}ens}$ in $L(p,q)$. We will recap some elementary notions from the theory of \textit{long knots} \cite{Budney2} prior to introducing this notion.

\subsection{Long knots in $\RR^3$}

\begin{definition}\label{long}
A \textbf{long embedding} $\gamma:\RR\to\RR^3$ is an embedding satisfying $\gamma(t)=(t, 0, 0)$ for $|t|\geq1$ and $\gamma(\RR)\cap\BB^3=\gamma([-1,1])$.
\end{definition}

We will often refer to the usual embeddings $\gamma:\NS^1\to\NS^3$ as \textbf{closed embeddings} in order to make the distinction from long embeddings explicit.

\begin{remark}
 In the strict sense, a long/closed knot is an oriented submanifold that can be parametrized by a long/closed embedding. Nonetheless, we will often abuse terminology and use the term knot to refer to embeddings as well whenever it is clear from the context.
\end{remark}

\begin{definition}
    We call the \textbf{standard part} of a long embedding to its restriction $\gamma(t)|_A$ to the subset $A=\{t\in\RR: |t|\geq 1\}$. Analogously, we say that the sub-arc $\gamma|_{[-1,1]}(t)$ is the \textbf{non-trivial part} of the long embedding. 
\end{definition}

Henceforth, we will write $\mathcal{O}p(X)$ to denote an arbitrarily small but non explicitly
specified open neighborhood of a set $X$.

\begin{remark}\label{LongKnotBall}
We can assume that, up to homotopy, a long embedding satisfies $\gamma(t)=(t,0,0)$ for $t\in\mathcal{O}p(\{-1,1\})$. Conversely, any embedding $\gamma:[-1,1]\to\RR^3$ that satisfies:
    \begin{itemize}
        \item[i)] $\gamma(-1,1)\subset\BB^3$ and
        \item[ii)] $\gamma(t)=(t,0,0)$ for $t\in\mathcal{O}p(\{-1,1\})$
    \end{itemize}
    
    is clearly the non-trivial part of a long embedding since it can be completed to an actual long embedding by defining $\gamma(t)=(t,0,0)$ for $|t|\geq 1$. Thus, from now on, we will use the terms long embedding/knot to refer to any embedding or embedded arc $\gamma:[-1,1]\to\RR^3$ satisfying $i)$ and $ii)$.
    \end{remark}

\begin{remark}\label{longtoclosed} It is a well known fact that isotopy classes of embeddings/knots in $\NS^3$ are in one-to-one correspondence with isotopy classes of long embeddings/knots in $\RR^3$.  This identification can be geometrically understood as follows. Taking the one-point compactification of  $\RR^3$ into $\NS^3$ allows to ``close'' a long knot into a knot in $\NS^3$. We can describe this process diagrammatically as ``taking the canonical closure'' (see below). Likewise, we can revert the process by taking the stereographic projection from $\NS^3$ to $\RR^3$. 
\end{remark}

For a rigorous proof of the statement in Remark \ref{longtoclosed}, we refer the reader to \cite[Thm 2.1]{Budney2}, where R. Budney provides an explicit decomposition of the space of knots in $\NS^3$ in terms of the space of long knots in $\RR^3$. This is much more general than what we will use in this work but in particular this decomposition implies the aforementioned correspondence.

Given a long knot $K$ in $\RR^3$, its canonically associated (closed) knot in $\NS^3$ can be described from its diagram. Indeed, working on a diagram, we can just join the endpoints $\gamma(-1)$, $\gamma(1)$ of the non-trivial part of the long knot by a smooth arc disjoint from the interior of $B^3$, whose projection onto the diagram does not yield self-intersections. We call this process, which is clearly well defined up to homotopy, ``taking the \textbf{canonical closure} of the long knot $K$''.

\begin{remark} Note that every isotopy of long knots readily defines, in a canonical way up to homotopy, an isotopy of knots in $\mathbb{S}^3$ by parametrically considering the canonical closure. Indeed, an isotopy of long knots works relative to the boundary of the $3$-ball $B^3$ and thus the whole isotopy lifts to an isotopy of (closed) knots in $\mathbb{S}^3$.
\end{remark}

If we regard $\NS^3\subset\mathbb{C}^2$, we can actually be more specific and also impose that, when taking the canonical closure, the basepoint of the knot $\gamma(0)$ must coincide with the north pole of $\NS^3$; i.e. $\gamma(0)=(1,0,0,0)\in\NS^3$ and, moreover, that its derivative must satisfy $\gamma'(0)=(0,1,0,0)$.  

 The space of long knots is homotopy equivalent to the space of closed knots with fixed basepoint and derivative at the basepoint; i.e. to the space
\begin{equation}\label{Emb*}
\mathfrak{Emb}_{*}=\{\gamma\in\mathfrak{Emb}(\NS^1,\NS^3): \gamma(0)=(1,0,0,0), \gamma'(0)=(0,1,0,0)\}.
\end{equation}
\begin{remark}\label{Remark:Emb*}
The fact that the space of long knots can be understood, up to homotopy, as the space $\mathfrak{Emb}_{*}$ of knots in $\NS^3$ with fixed basepoint and derivative at the basepoint is a standard result. A rigorous proof can be derived, e.g. from \cite[Prop 1.4]{BC}. 
\end{remark}

  \begin{figure}[h]
            \centering
            \includegraphics[width=0.7\linewidth]{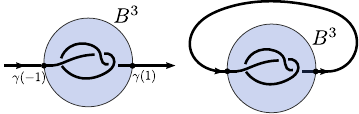}
            \caption{The Figure describes the process of taking the canonical closure (right) of a long knot (left) by attaching a smooth arc disjoint from the interior of $B^3$ to the non-trivial part of the long knot, whose projection onto the diagram does not yield self-intersections. This is well defined up to homotopy and canonically yields a knot in $\mathbb{S}^3$ out of any long knot in $\RR^3$.}
            \label{fig:canonicalclosure}
        \end{figure}

\begin{remark}\label{rem:complementball}
 Let us provide another interpretation of this discussion. Given an embedding $\gamma:\NS^1\to\NS^3$, choose a point $p$ in the trace of $\gamma$ and an arbitrarily small ball $B$ containing $p$ in its interior and such that $\partial B^3$ intersects the trace of $\gamma$ exactly in two antipodal points (Figure \ref{fig:BallLongKnot2}). Then the part of $\gamma$ outside $B$ can be understood as a long embedding $\gamma_{Long}(t):\RR\to\RR^3$. Indeed, note that $\NS^3\setminus B$ is a topological $3$-ball where $B$ can be understood as a small neighborhood of the point $\infty\in\NS^3$ from which we stereographically project in order to obtain the long knot. Therefore, $\NS^3\setminus B$ contains the complementary part of the knot (i.e. everything but the part contained in $B$) so that, up to homotopy, yields an embedding as in Remark \ref{LongKnotBall}. From this perspective, taking the canonical closure can be reinterpreted (up to homotopy) as attaching to the long knot the small arc that lies within $B$.

 Furthermore, given a diagram of the knot in a plane $\Pi\subset\RR^3\subset\NS^3$, there is a natural way of passing from the closed knot (together with the choice of small ball $B$) to the associated long knot. Indeed, up to a rigid movement we can identify $\Pi\simeq \RR^2$ and just consider the knot very close to the long axis $\beta(t)=(t,0,0)$. Then, by removing a tiny arc around the point $p$ of $\gamma$ and removing a tiny arc around the origin in the long axis $\beta(t)$, we can just join the remaining components from both curves without creating new self-intersection points, yielding the long embedding associated to $\gamma$ (Figure \ref{fig:BallLongKnot2}). This long embedding is clearly the one associated to $\gamma$ since the removed segment can be just regarded, up to homotopy, as its canonical closure.
\end{remark}

  \begin{figure}[h]
            \centering
            \includegraphics[width=0.65\linewidth]{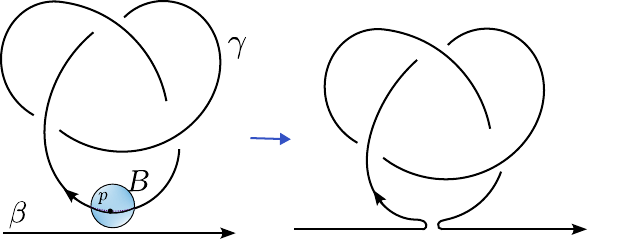}
            \caption{The figure on the left shows a diagram of a closed embedding $\gamma:\NS^1\to\NS^3$ together with a choice of a point $p$ on its trace. $B$ is an arbitrarily small ball  containing $p$ on its interior such that $\partial B$ intersects the trace of $\gamma$ in exactly two antipodal points. The part of $\gamma$ outside $B$ is a long embedding $\gamma_{Long}(t):\RR\to\RR^3$ according to Remark \ref{rem:complementball}. The figure on the right shows how to obtain a diagrammatic representation of $\gamma_{Long}(t)$ just by removing the segment around $p$ and connecting it to the long axis $\beta(t)=(t,0,0)$.}
            \label{fig:BallLongKnot2}
        \end{figure}

As a consequence of the discussion above and, in particular, from the correspondence between isotopy classes of knots and long knots, we get the following result.

\begin{proposition}\label{IndependanceLongKnot}
Let $\gamma_1:\RR\to\RR^3$ and $\gamma_2:\RR\to\RR^3$ be two long knots so that their canonical closures $\tilde\gamma_1:\NS^1\to\NS^3$ and $\tilde\gamma_2:\NS^1\to\NS^3$ are isotopic as knots in $\NS^3$. Then, $\gamma_1$ and $\gamma_2$ are isotopic as long knots. 
\end{proposition}

\begin{proof}
    The result is a direct consequence of the correspondence between isotopy classes of knots and long knots. See \cite[Thm 2.1]{Budney2}.
\end{proof}

\subsection{On non-local knots in $L(p,q)$}

This subsection contains a brief discussion about non-local knots in $L(p,q)$. 

\begin{definition}
    We say that a \textbf{local knot} in a $3$-manifold $M$ is a knot that lies within an embedded $3$-ball in $M$. On the contrary, knots that do not lie within any embedded $3$-ball in $M$ are called \textbf{non-local knots}.
\end{definition}

\begin{remark}
    If $M=\RR^3$ or $M=\NS^3$ then every knot in $M$ is local. Nonetheless, this is not necessarily the case for more general manifolds $M$. Local knots are called \textbf{contractible knots} in \cite{viro} and, moreover, local knots in the particular case of $\RR\PP^3$ are also called \textbf{affine knots} (\cite{Drobotukhina1990}, \cite{Mishra2023}, \cite{Kauffman}).
\end{remark}
Moreover, in the case of lens spaces $L(p,q)$ we can distinguish several types of non-local knots depending on the class that they represent in homology, $H_1(L(p,q))=\ZZ_p$.

\begin{definition}
  We say that a knot $K\subset L(p,q)$ is of class-$n$ (where $n\in\{0,1,\cdots, p-1\}$) if it represents the class $$[n]\in H_1(L(p,q))=\ZZ_p=\{[0],[1],\cdots, [p-1]\},$$
and where the generator of $H_1(L(p,q))$ in the $3$-ball model (Figure \ref{fig: generatorH1}) represents the class $[1]\in H_1(L(p,q))$. \end{definition}

\begin{remark}
    Every local knot is of class-$0$ but the converse does not necessarily hold; i.e. there exist nullhomologous non-local knots (see e.g. \cite[Fig. 4]{Narayanan2025}). On the other hand, every knot of class-$n$ where $n\neq 0$ is necessarily non-local.
\end{remark}

\subsubsection{A canonically defined non-local unknot in $L(p,q)$}
Let us describe a particular instance of a non-local knot in every lens space $L(p,q)$. This knot can be easily described in the $3$-ball model of $L(p,q)$ and, moreover, will be key in further constructions.

\begin{definition}
    Take a point $\mathfrak{p}\in\partial B^3/\sim$ that does not lie in the equatorial disk and consider a segment that connects this point with its identified point $f_z\circ g_{p,q}(\mathfrak{p})$ (Definition \ref{def:lensspace}) (see Figure \ref{fig:nonlocalunknot}). 
Note that this arc yields (up to isotopy) an embedding in $L(p,q)$, since it represents a closed curve. This knot is called the \textbf{class-$q$ non-local unknot}.
    
\end{definition}
 Clearly, this definition does not depend on the choice of point $\mathfrak{p}$ (since $\partial B^3$ is connected) and it is thus well defined up to isotopy. Moreover, this knot represents  the class $[f^q]$ in $\pi_1(L(p,q))=\ZZ_p$, where recall that $f$ denotes the generator of $\pi_1(L(p,q))$ given by taking a $p$-th fraction of the equatorial curve in $B^3/\sim$ (Subsection \ref{generatorH1}). This knot is of class-$q$. This readily follows from its definition (see Figure \ref{fig:nonlocalunknot} for the visualization of why it represents the class $[f^q]$ in $\pi_1(L(p,q))$) or check \cite[Lemma 4.3]{manfredi2014knots} from where the same fact readily follows.

\begin{figure}[h!]
    \centering
    \includegraphics[width=0.7\linewidth]{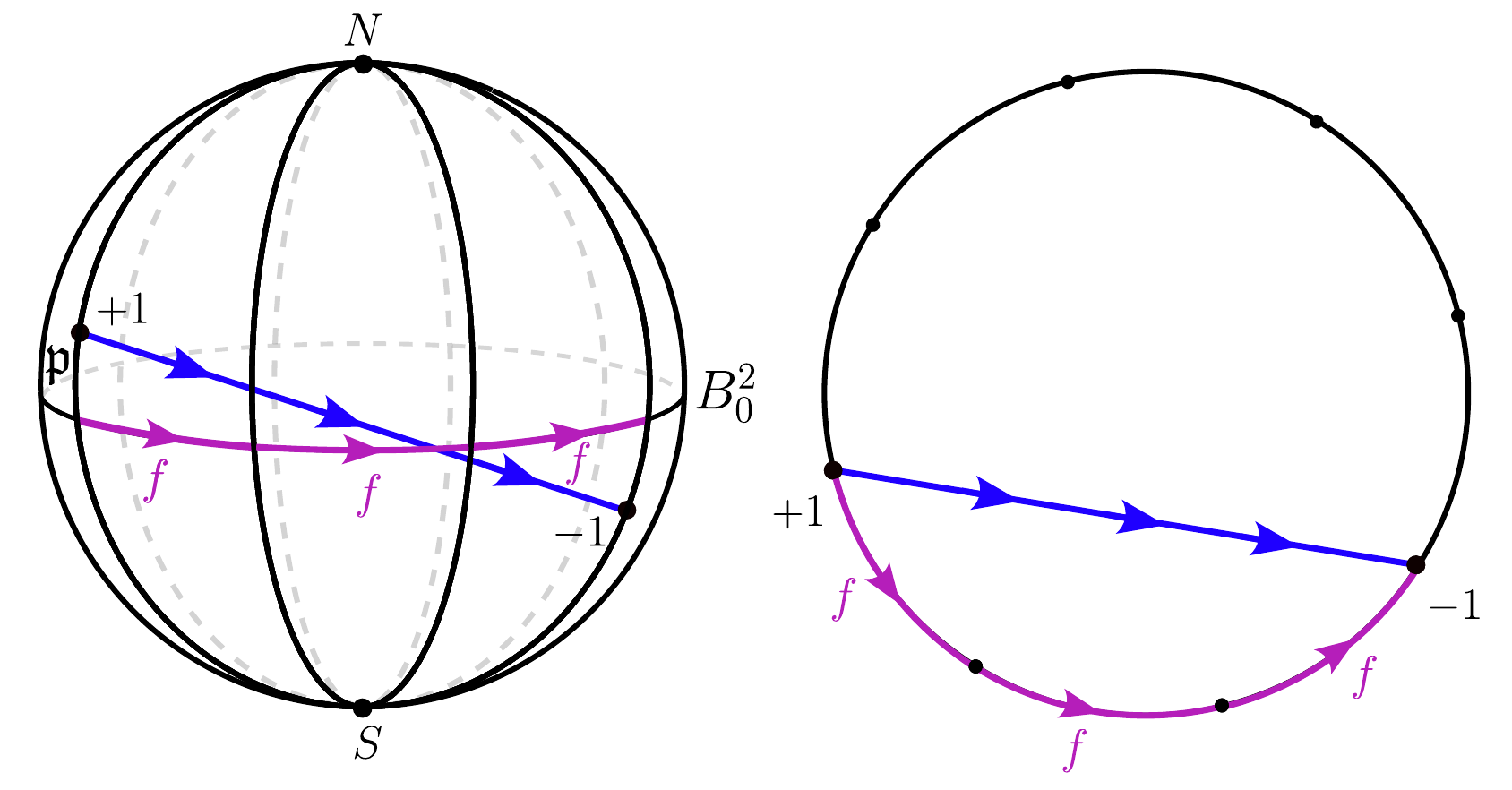}
    \caption{Depiction of the class-$q$ non-local unknot (left) in a lens space $L(p,q)=B^3/\sim$ for the particular choice of $(p,q)=(8,3)$ together with its associated disk diagram (on the right). This knot is defined by joining two identified points $\mathfrak{p}$ and  $f_z\circ g_{p,q}(\mathfrak{p})$ in $\partial B^3\setminus\partial B_0^2$ via a segment and it represents the homotopy class $[f^q]\in\pi_1(L(p,q))$. Since $\gcd(p,q)=1$, it turns out that its homotopy class generates both the whole fundamental group $\pi_1(L(p,q))=\ZZ_p$ and also the homology group $H_1(L(p,q))=\ZZ_p$.}
    \label{fig:nonlocalunknot}
\end{figure}

\subsection{From long knots in $\RR^3$ to lensifications in $L(p,q)$}\label{subs: fromlongtolens}

In this subsection we will discuss how starting from a given long knot in $\RR^3$ we can construct, in a natural way, an associated non-local knot in any lens space $L(p,q)$.

First, observe that for $u\in(0,1)$ the diffeomorphism $\phi_u(x,y,z)=(u\cdot x, u\cdot y, u\cdot z)$ shrinks the $3$-ball $B^3$ onto a smaller ball $B_u=\{x\in\RR^3: |x|\leq u\}$ of radius $u$. 

    Let $\gamma_{\mathcal{L}ong}:[-1,1]\to\RR^3$ be a long embedding (Remark \ref{longtoclosed}), where therefore $\gamma_{\mathcal{L}ong}(-1)=(-1,0,0)$ and $\gamma_{\mathcal{L}ong}(1)=(1,0,0)$. 

    \begin{definition}\label{shrinking}
        Let $u\in(0,1)$. We call the \textbf{$u$-shrinking} of $\gamma_{\mathcal{L}ong}$ to the embedding $\gamma_{\mathcal{L}ong}^u:[-1,1]\to\RR^3$ defined as: \[\gamma_{\mathcal{L}ong}^u:= \phi_u\circ\gamma_{\mathcal{L}ong}.\]

    \end{definition}

\subsubsection{Definition of the lensification of a knot}

Start with the class-$q$ non-local unknot parametrized as $\mathfrak{u}(t):\NS^1\to\ B^3/\sim$ with $\mathfrak{u}(-1)\sim\mathfrak{u}(1)\in\partial B^3$ (where we regard $\NS^1=[-1,1]_{-1\sim 1}$).
Apply a small isotopy to the class-$q$ non-local unknot so that it becomes flat in the region $t\in(-\frac{1}{2},\frac{1}{2})$; i.e. $\mathfrak{u}(t)=(t,0,0)$, for $t\in(-\frac{1}{2},\frac{1}{2})$ and so that it does not contain any self-intersection point in its associated disk diagram (see Figure \ref{fig:flattening}). Abusing notation, we will still denote this new ``flattened'' embedding by $\mathfrak{u}(t):\NS^1\to\ B^3/\sim$.

\begin{figure}[h!]
    \centering
    \includegraphics[width=0.7\linewidth]{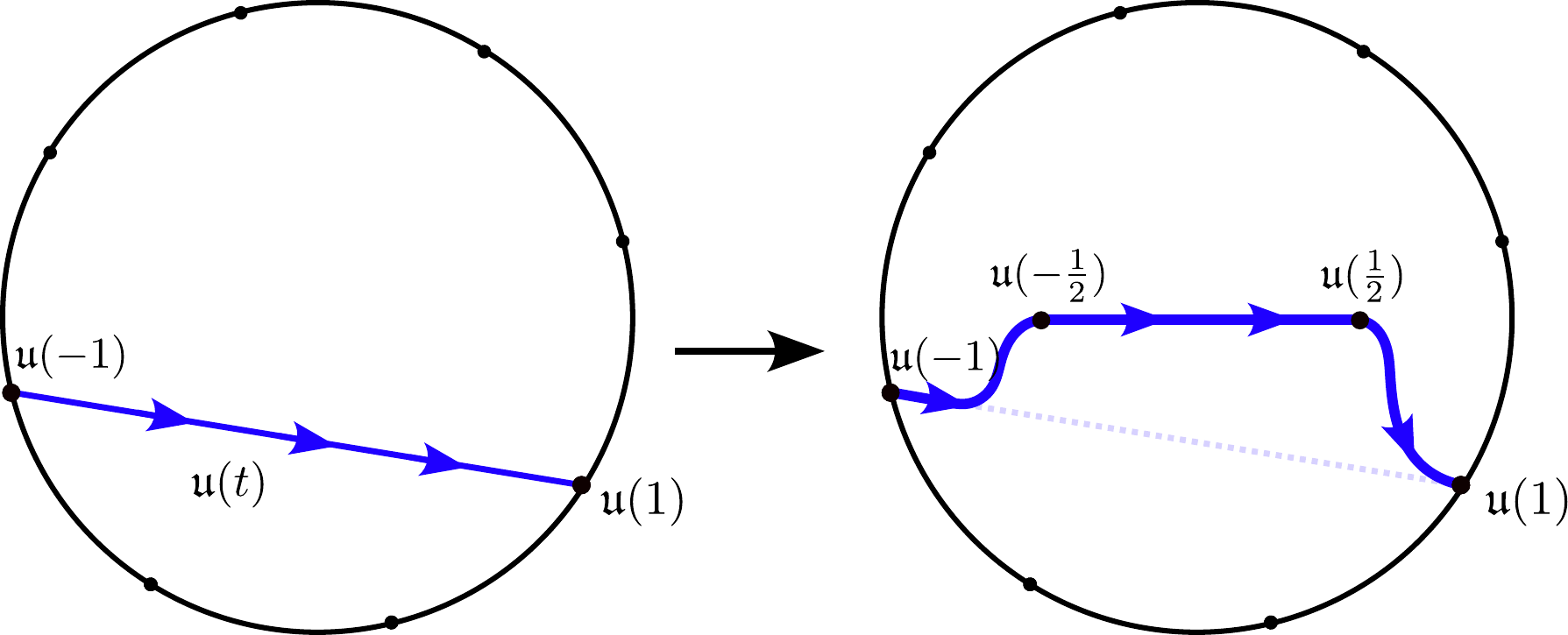}
    \caption{Visualization in a disk diagram of the deformation (through an isotopy) of the class-$q$ non-local unknot $\mathfrak{u}(t):\NS^1\to\ B^3/\sim$ so that it becomes flat in the region $t\in(-\frac{1}{2},\frac{1}{2})$; i.e. $\mathfrak{u}(t)=(t,0,0)$, for $t\in(-\frac{1}{2},\frac{1}{2})$ without introducing self-intersections in its associated disk diagram.}
    \label{fig:flattening}
\end{figure}

On the other hand, take the long knot $\gamma_{\mathcal{L}ong}$ and consider its $(1/2)$-shrinking $\gamma^{1/2}_{\mathcal{L}ong}(t)$. Replace now the flat portion of $\mathfrak{u}(t)$, i.e. the sub-arc corresponding to $t\in(-\frac{1}{2},\frac{1}{2})$, by $\gamma^{1/2}_{\mathcal{L}ong}(2t)$ (Figure \ref{fig:lensification_cons}); i.e. the resulting embedding $\gamma_{Lens}(t)$ after this replacement can be explicitly described as:
\begin{equation}
\label{eq:lensification}
\gamma_{Lens}(t) =
\begin{cases}
\mathfrak{u}(t) & \text{if } t \in [-1,1] \setminus (-\frac{1}{2},\frac{1}{2}), \\[6pt]
\gamma^{1/2}_{\mathcal{L}ong}(2t) & \text{if } t \in (-\frac{1}{2},\frac{1}{2}).
\end{cases}
\end{equation}

    \begin{figure}[h!]
    \centering
    \includegraphics[width=0.7\linewidth]{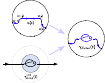}
    \caption{Construction of the lensification $\gamma_{Lens}(t)$ of a trefoil knot in a lens space $L(p,q)$ (for the particular choice of $(p,q)=(8,3)$). The top-left frame depicts the class-$q$ non-local unknot after a deformation that made it flat in the region $t\in(-\frac{1}{2},\frac{1}{2})$. The bottom-left frame depicts the associated long knot after a $\frac{1}{2}$-scale shrinking. The frame on the right shows the lensification $\gamma_{Lens}(t)$ built out of the previous two knots.}
    \label{fig:lensification_cons}
\end{figure}

   \begin{definition}\label{def: lensification}
       Let $\gamma:\NS^1\to\NS^3$ be an embedding. Take a point $p$ in the trace of $\gamma$ and an arbitrarily small ball  $B$ containing the point and such that $\partial B$ intersects the trace of $\gamma$ in two antipodal points. Recall from Remark \ref{rem:complementball} that the part of the embedding outside $B$ determines a well defined long embedding $\gamma_{\mathcal{L}ong}:\RR\to\RR^3$ with a precise diagrammatic description (Remark \ref{rem:complementball}). We then define the  \textbf{lensification} of $\gamma$ as the embedding $\gamma_{Lens}(t):\NS^1\to L(p,q)$ defined by Equation (\ref{eq:lensification}).
   \end{definition} 
   
 Choosing different points $p$ in the trace of $\gamma$ could a priori yield potentially different associated lensifications $\gamma_{Lens}(t):\NS^1\to L(p,q)$. Theorem \ref{independence} will show that this choice is unique up to isotopy.
 
\begin{remark}\label{rem:diagrammaticLensification}
Note that for each choice of point $p$ in the trace of $\gamma$ and once a diagram of $\gamma$ is given, then the associated lensification $\gamma_{Lens}(t):\NS^1\to L(p,q)$ (Def. \ref{def: lensification}) possesses a concrete diagrammatic description. Working in a disk diagram, take the class-$q$ non-local unknot $\mathfrak{u}(t)$ passing very close to the point $p$. Remove a small arc from both knots and join the remaining components without introducing new self-intersection points. In view of Definition \ref{def: lensification} and Eq. (\ref{eq:lensification}), this diagrammatic description yields the lensification of $\gamma$. In other words, Definition \ref{def: lensification} recovers the diagrammatic description in $\RR\PP^3$ given in \cite{Kauffman} (see e.g. Figures 33 to 36 in \cite{Kauffman}).
\end{remark}

\begin{remark}\label{rem:conSumLocal} We would like to note that the lensification of a knot (and thus the projectivization as a particular instance) can alternatively be defined as a connect-sum operation for pairs as described by D. Rolfsen in \cite[p. 40, Sec. 2.G]{rolfsen2003knots}. Given a knot $W\subset L(p,q)$ and a local knot $K\subset \NS^3$, we can consider the connected sum  of pairs $(L(p,q), W)\#(\NS^3,K)$. Note that this connected sum is well defined and thus does not depend on the choice of points. This is a consequence of the Annulus Theorem (see \cite[pp. 41--42, Ex. 2 \& 3]{rolfsen2003knots}). The fact that the connected sum of a local knot in $L(p,q)$ with any other knot is well defined is also implicitly considered in the PhD thesis of E. Manfredi \cite[p. 22]{manfredi2014knots}. 

Following this approach would yield an affirmative answer to the question about the equivalence between projectivizations of the same affine knot from a purely theoretical perspective. Nonetheless, even if this idea guarantees the \textit{existence} of an isotopy, it is not constructive. The goal of this note is not only to answer the question in the affirmative but rather to provide an explicit way of constructing isotopies between any two given lensifications of the same knot. We thus define the lensification operation building upon the notion of \textit{long knots}. This, in turn, will allow us to describe an algorithm that determines an explicit isotopy between any two given lensifications of the same knot. We will make use of an idea of A. Hatcher \cite{hatcher3} developed in the realm of embedding spaces.
\end{remark}

\begin{theorem}\label{independence}
Let $\gamma:\NS^1\to\NS^3$ be an embedding. All possible lensifications of $\gamma$ in $L(p,q)$ are isotopic. In other words, if two lensification  $\gamma_{\mathcal{L}ens}^1, \gamma_{\mathcal{L}ens}^2$ were produced out of two different long knots $\beta_1,\beta_2:\RR\to\RR^3$, both associated to $\gamma$, respectively, then $\gamma_{\mathcal{L}ens}^1$ and $\gamma_{\mathcal{L}ens}^2$ are isotopic in $L(p,q)$. \end{theorem}

\begin{proof}
By hypothesis, the long knots $\beta_1$ and $\beta_2$  yield two knots isotopic to $\gamma$ when considering their canonical closures. This means that, according to the bijective correspondence between isotopy classes of knots in $\NS^3$ and long knots in $\RR^3$ (Remark \ref{longtoclosed}), the associated long knots $\beta_1$ and $\beta_2$ are thus isotopic within the space of long knots. Therefore, there exists a $1$-parametric smooth family of long knots $\left(\beta_s\right)_{s\in[1,2]}$ that connects $\beta_1$ with $\beta_2$. 

Define now the associated $1$-parametric family of $(1/2)$-shrinkings (Definition \ref{shrinking}); i.e. $\left(\beta^{1/2}_s\right)_{s\in[1,2]}:=(\phi_{\frac{1}{2}}\circ\beta_s)_{s\in[1,2]}$. Observe now that, since the $1$-parametric family $\left(\beta_s\right)_{s\in[1,2]}$ takes place within the space of long knots, the new family $\left(\beta^{1/2}_s\right)_{s\in[1,2]}$ satisfies two conditions:

\begin{itemize}
    \item[i)] $\beta^{1/2}_s(-\frac{1}{2})=(-\frac{1}{2},0,0)$ and $\beta^{1/2}_s(\frac{1}{2})=(\frac{1}{2},0,0)$ for every $s\in[1,2]$. 
    \item[ii)] $\beta_s^{1/2}\left([-\frac{1}{2},\frac{1}{2}]\right)\subset B_{\frac{1}{2}}$; i.e. the non-constant part of the isotopy lies within the ball of radius $\frac{1}{2}$.
\end{itemize}

This readily defines an isotopy of lensifications between $\gamma_{\mathcal{L}ens}^1$ and $\gamma_{\mathcal{L}ens}^2$. Indeed, it suffices to consider the following $1$-parametric family of knots:

\begin{equation}
\label{eq:lensificationPARAMETRIC}
\gamma_{s}(t) =
\begin{cases}
\mathfrak{u}(t) & \text{if } t \in [-1,1] \setminus (-\frac{1}{2},\frac{1}{2}), \\[6pt]
\beta^{1/2}_{s}(2t) & \text{if } t \in (-\frac{1}{2},\frac{1}{2}).
\end{cases}
\end{equation}
Therefore, the proof is complete. 
\end{proof}

\subsection{Adapting A. Hatcher's idea to lensifications in $L(p,q)$}\label{fox-hatcher}
In this last section we will first describe the original idea from A. Hatcher's work \cite{hatcher3} and then we will show how to export it to the setting of lensifications of knots in lens spaces.

\subsubsection{The Fox-Hatcher loop}\label{sec:algorithm}

In \cite{hatcher3}, A. Hatcher discusses global homotopical properties of the space of knots in the $3$-sphere. Among other ideas, he describes a loop in the space of long knots; i.e. an isotopy of long knots that starts and ends in the same long knot. This loop is commonly known as the Fox-Hatcher loop. It was actually first introduced in \cite{Fox1966Rolling} in 1966 by R. Fox and later studied more extensively by A. Hatcher \cite{hatcher3} in 2002. He shows that this loop yields a non-trivial element in the fundamental group of any long knot space except for the long unknot-component. We will follow the exposition from the work of S. Kanou and K. Sakai \cite{KanouSakai2023}. See also the work of T. Fiedler \cite{Fiedler2019}.

Let us introduce some notions first. Start with a parametrization of a knot $\gamma:\NS^1\to\NS^3\subset\mathbb{C}^2$ equipped with a framing $\omega(t)$; i.e. $\omega(t)$ denotes an orthonormal basis of $T_{\gamma(t)}\NS^3$, of the form $\omega(t)=\left(\frac{\gamma'(t)}{||\gamma'(t)||},v_2(t),v_3(t)\right)$. Additionally, assume without loss of generality that the knot is parametrized by arclength; i.e. $||\gamma'(t)||=1$ for all $t\in\NS^1\simeq [0,2\pi]$ and thus $\omega(t)=\left(\gamma'(t),v_2(t),v_3(t)\right)$. 

Consider the following loop of matrices 
\[
A_\theta:=\left( \gamma(\theta) \, \middle| \, \gamma'(\theta)  \, \middle|   v_2(\theta)\, \, \middle| \, v_3(\theta) \right)\in\SO(4),\quad \theta\in\NS^1\simeq [0,2\pi],
\]

and define now the following loop of knots:
\begin{equation}\label{stereo}
\gamma_\theta(t):=(A_\theta)^{-1}\cdot\gamma(t+\theta), \quad \theta\in\NS^1\simeq [0,2\pi].
\end{equation}

Observe that this loop of knots satisfies two properties:

\begin{itemize}
    \item[$i)$] $\gamma_\theta(0)=(A_\theta)^{-1}\cdot\gamma(\theta)=(1,0,0,0)$,
    \item[$ii)$] $\frac{d}{dt}\left(\gamma_\theta(t)\right)|_{t=0}=(A_\theta)^{-1}\cdot\gamma'(\theta)= (0,1,0,0)$.
\end{itemize}

Condition $i)$ codifies the fact that the basepoint of the knot remains fixed all along the loop and, furthermore, coincides with the north pole of $\NS^3$; i.e. with $N=(1,0,0,0)\in\mathbb{S}^3$. Analogously, condition $ii)$ codifies the fact that the derivative at the basepoint of the knot remains fixed (since it coincides with $(0,1,0,0)\in\NS^3$) through the whole loop as well. This yields a loop of long knots in $\RR^3$ (Remark \ref{Remark:Emb*}) built out of a framed knot in $\NS^3$. This is often called the \textbf{Fox-Hatcher} loop in the literature.

\begin{remark}\label{rem:koytcheff}
There is a nice geometric interpretation of this loop \cite{hatcher3}. Place a small bead (i.e. a small $3$-ball $B$) on the closure of a framed knot in $\NS^3$ (recall Remark \ref{rem:complementball}). The part of the knot outside this bead; i.e. on $\NS^3\setminus B$ yields a long knot (Remark \ref{rem:complementball}). Moving this bead all along the knot (as if it were a wire supporting it) while respecting the framing  yields the Fox-Hatcher loop once the bead returns to its initial position.

More rigorously, as noted in \cite{HavensKoytcheff2021}, this process can be equivalently understood as moving the point $\infty\in\NS^3$ all along the trace of $\gamma$ and simultaneously taking the stereographic projection onto $\RR^3$ while respecting the given framing. This is in fact equivalent to the expression in Eq. (\ref{stereo}), where $\infty=(1,0,0,0)\in\NS^3$. 
\end{remark}

\subsubsection{The blackboard framing case}

When we have a regular projection of a knot, we can take the blackboard framing associated to it. As such, the previous construction takes an immediate diagrammatic description. The effect of applying the loop from Eq. (\ref{stereo}) to $\gamma$ is clear from a diagrammatic viewpoint. The corresponding loop of knots $\gamma_\theta(t)$  (see Figure \ref{fig:sliding}) has constant evaluation at the origin $\gamma_\theta(0)$ (black marked point in the six frames) with constant derivative $\frac{d}{dt}(\gamma_\theta(t))|_{t=0}$ (black arrow). Additionally, since the framing is the blackboard framing (i.e. a normal vector to the projection plane determines this framing), this can be regarded, up to homotopy, as a loop where the knot slides rigidly over itself pivoting on its fixed basepoint with derivative, as is shown in Figure \ref{fig:sliding}.

\begin{figure}[!h]
    \centering
    \includegraphics[width=0.6\linewidth]{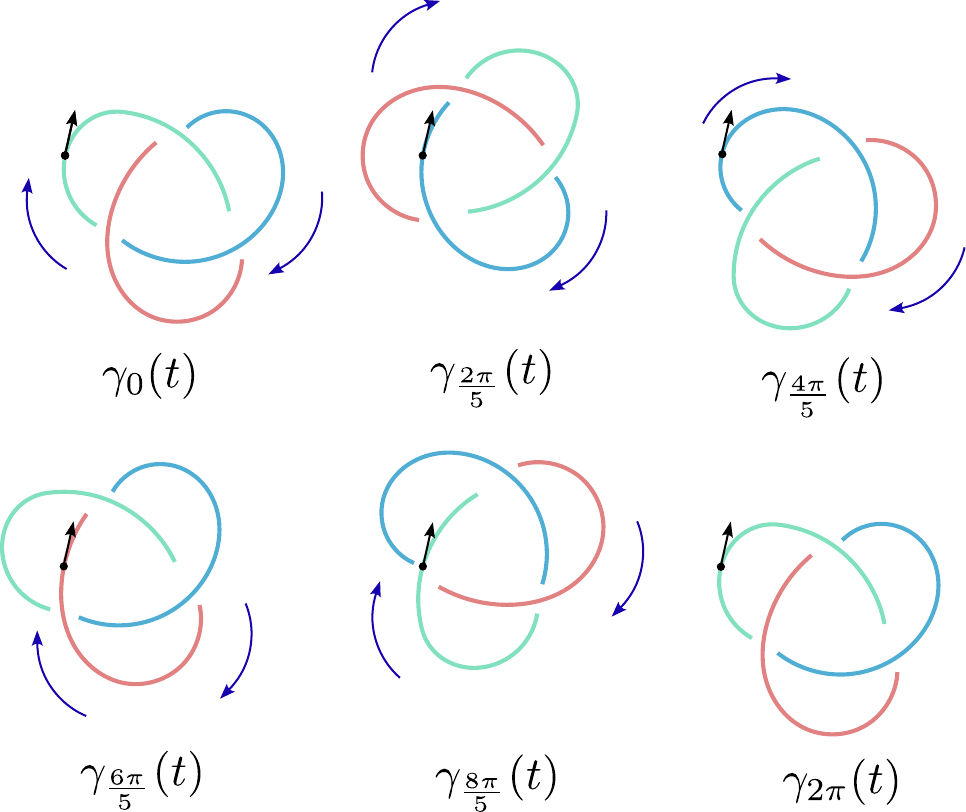}
    \caption{Given a knot $\gamma$ with a diagram, regard it as a framed knot equipped with the blackboard framing. The corresponding loop of knots $\gamma_\theta(t)$ (Eq. (\ref{stereo})) has constant evaluation at the origin $\gamma_\theta(0)$ (black marked point in the six frames) with constant derivative $\frac{d}{dt}(\gamma_\theta(t))|_{t=0}$ (black arrow). It can thus be regarded as a loop where the knot slides rigidly over itself with fixed basepoint and derivative. We have coloured some branches of the knot in order to provide a more clear intuition.}
    \label{fig:sliding}
\end{figure}

\begin{remark}\label{beadBlackboard}
We can therefore obtain in a straightforward fashion as well the corresponding loop of long knots in a diagrammatic way following Remark \ref{rem:complementball}. Indeed, note that this is equivalent to the loop described by moving the bead (as in Remark \ref{rem:koytcheff}) all along the knot following the blackboard framing, i.e. without introducing further rotations with respect to the projection plane. In other words, each knot $\gamma_{\theta}$ in the loop (for $\theta\in\NS^1$) can be diagrammatically represented as the long knot arising by removing a small segment around $\gamma_\theta(0)$ and gluing it to the long axis in the obvious manner.
\end{remark}

There is a good diagrammatic description due to A. Hatcher \cite{hatcher3, KanouSakai2023} of the associated loop of long knots; i.e. in the case where the chosen framing is the blackboard framing. Start with the diagram of the long knot and, following the first branch according to the orientation of the long knot (from left to right), focus on its first crossing. There are two possibilities.
\begin{itemize}
    \item[{{Case 1}}.] It represents an upper branch; i.e. it passes \textbf{over} the other branch. Then pull the corresponding underpass all under the knot so that it yields an underpass at the rightmost part of the diagram.

    \item[{{Case 2}}.] It represents a lower branch; i.e. it passes \textbf{under} the other branch. Then pull the corresponding overpass all over the knot so that it yields an underpass at the rightmost part of the diagram.
\end{itemize}

The moves we just described in Case 1 and Case 2 are called Fox-Hatcher moves or, in short, \textbf{HF-moves} in \cite[Sec 2.]{KanouSakai2023}. Note that the process we just described transforms, one by one, each crossing on the left part of the diagram into a crossing on the right part of the diagram. If the total number of crossings in the diagram is $k\in\mathbb{N}$, then the process describes the Fox-Hatcher loop associated to the blackboard framing after applying $k$ steps. We will import this idea into the setting of lens spaces by passing from diagram of long knots in $\RR^3$ to disk diagrams of lensifications of knots in $L(p,q)$ (recall Subsection \ref{subs: fromlongtolens}).

\subsubsection{The algorithm for lensifications in $L(p,q)$.}\label{algorithm}

Consider a given diagram of a knot $\gamma:\NS^1\to\NS^3$ and two different lensifications $\gamma_{Lens}^1(t), \gamma_{Lens}^2(t)$. Take a disk diagram of each of these two projective knots exactly as explained in Remark \ref{rem:diagrammaticLensification}; i.e. the corresponding diagrams can be described by taking, respectively, points $p_1, p_2$ on the trace of $\gamma$, removing a small segment containing such point and then connecting the remaining part with the class-$q$ unknot $\mathfrak{u}(t)$. Additionally, note that, outside the $\frac{1}{2}$-radius disk, both knot diagrams coincide with $\mathfrak{u}(t)$ there. 

Since the two lensifications  $\gamma_{Lens}^1(t), \gamma_{Lens}^2(t)$ are constructed by removing a small arc (before connecting it to the class $q$-unknot) around two different points $p_1, p_2$ of the same knot, the associated long knots are related by sliding the basepoint along the knot from $p_1$ to $p_2$. This is precisely what the Fox-Hatcher loop does, which can be read directly in a diagram under the choice of the blackboard framing. The loop can be described by an application of a finite number of HF-moves to the long knots within the $\frac{1}{2}$-radius disk following the process. In other words, after a finite sequence of those moves, the roles of $p_1$ and $p_2$ are exchanged and we thus get an isotopy between the associated lensification  $\gamma_{Lens}^1(t)$ and $\gamma_{Lens}^2(t)$. Let us isolate the formulation of this algorithm as follows.

Start with a disk diagram of a lensification in $L(p,q)$ and assume without loss of generality that its diagram coincides with the class-$q$ unknot outside the radius $\frac{1}{2}$-disk. Starting with the first branch according to the orientation of the knot, focus on its first crossing. There are two possible cases.
\begin{itemize}
    \item[{{Case 1}}.] It represents an upper branch; i.e. it passes \textbf{over} the other branch in the crossing. Then pull the corresponding underpass all under the knot so that it yields an underpass at the rightmost part of the diagram within the radius $\frac{1}{2}$-disk (this is represented in Figure \ref{fig:algoOver}). This movement must take place within such disk, it cannot escape it.

\begin{figure}[!h]
    \centering
    \includegraphics[width=0.85\linewidth]{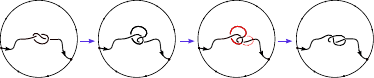}
    \caption{\textbf{Case $1$ in the algorithm}: the first branch according to the orientation of the knot in the disk diagram represents an upper branch in the first crossing. The algorithm specifies to pull the corresponding underpass (in red) all under the knot so that it yields an underpass in the other part of the diagram. The picture depicts this situation in the particular case of $L(8,3)$.}
    \label{fig:algoOver}
\end{figure}

    \item[{{Case 2}}.] 
    It represents a lower branch; i.e. it passes \textbf{under} the other branch. Then pull the corresponding overpass all over the knot so that it yields an overpass at the rightmost part of the diagram within the radius $\frac{1}{2}$-disk (this is represented in Figure \ref{fig:algoUnder}). This movement must take place within such disk, it cannot escape it.  
    
    \begin{figure}[!h]
    \centering
    \includegraphics[width=0.85\linewidth]{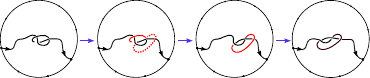}
    \caption{\textbf{Case $2$ in the algorithm}: the first branch according to the orientation of the knot in the disk diagram represents a lower branch in the first crossing. The algorithm indicates to pull the corresponding overpass (in red) all over the knot so that it yields an overpass in the other part of the diagram. Once again, this diagram corresponds to a lensification in $L(8,3)$.}
    \label{fig:algoUnder}
\end{figure}

\end{itemize}

The reader should note that all the moves in the algorithm take place in the interior part of the disk diagram; thus they are valid moves that yield actual isotopies between knots in $L(p,q)$ according to \cite[Corollary 3]{CattabrigaManfredi2018}.

\section{Applications of the algorithm}\label{sec: applications}

In this last section we will apply the algorithm to several pairs of projective knots from the literature \cite{Kauffman, Narayanan2025} whose equivalence has been raised as an open question. Each such pair is formed by two different projectivizations of the same given knot. In particular, if we apply the algorithm to each of them separately, we know that we will eventually reach the same given projective knot. This will thus provide an explicit isotopy.

\subsection{First application}\label{firstapplication}

The work \cite{Kauffman} states that the equivalence problem for the projective knots $K_1$ (top-left in Figure \ref{fig:pair2}) and $K_2$ (top-right in Figure \ref{fig:pair2}) is unknown: ``\textit{The paper ends with problems about these approaches and an example of multiple projectivizations of the figure-$8$ knot whose equivalence is unknown at this time}''. As an application of the present work, we will be able to provide an explicit isotopy for such multiple projectivizations of the figure-$8$ knot (Figure \ref{fig:pair2}), thus showing their equivalence. 

Although all arrows in Figure \ref{fig:pair2} are reversible (the figure describes an isotopy), we suggest the reader to follow the order indicated by the arrows for a better understanding of the sequence. Note, additionally, that the original pair of knots in \cite{Kauffman} were not oriented, but we have equipped them with an arbitrary orientation in order to apply the algorithm. Indeed, any isotopy of oriented knots readily lifts to an isotopy of unoriented knots. Furthermore, had we chosen the opposite orientation, the algorithm would yield a valid isotopy as well.

\begin{figure}[!h]
    \centering
    \includegraphics[width=0.85\linewidth]{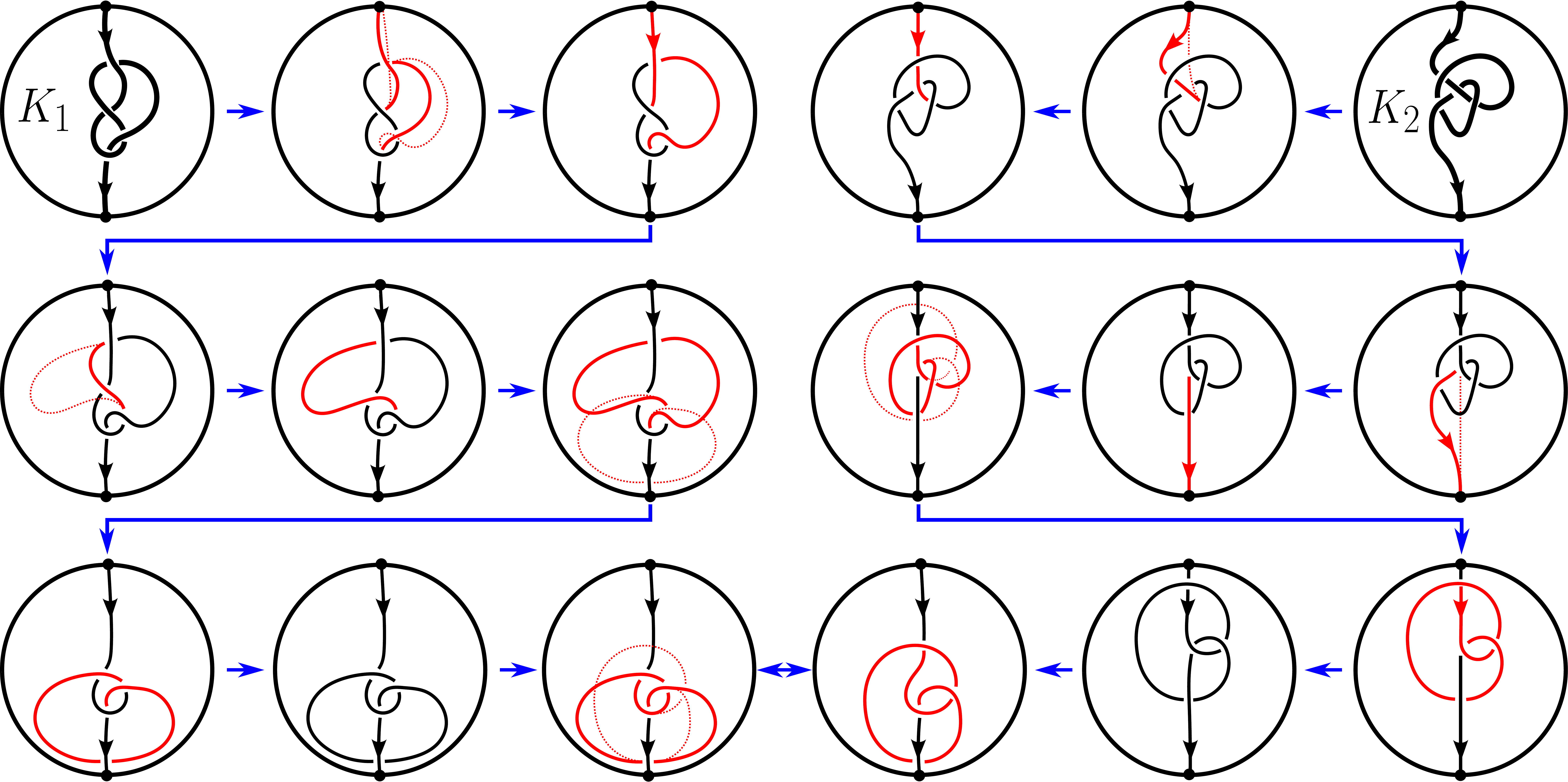}
    \caption{The problem of whether the multiple projectivizations of the figure-$8$ knot $K_1$ (top-left) and $K_2$ (top-right) are equivalent has been raised in \cite{Kauffman}. The figure depicts an application of the algorithm, which produces an explicit isotopy between the two and thus provides a constructive solution to the problem. Although all arrows are reversible (the figure describes an isotopy), we suggest the reader to follow the order indicated by the arrows for a better understanding of the sequence.}
    \label{fig:pair2}
\end{figure}


\subsection{Second application}

The question of whether the projectivizations $K_1$ (top-left in Figure \ref{fig:pair3}) and $K_2$ (bottom-right in Figure \ref{fig:pair3}) are isotopic was proposed by V. Narayanan in  \cite{Narayanan2025}. An application of the algorithm from Section \ref{algorithm} yields an explicit isotopy that we have depicted in Figure \ref{fig:pair3}, thus providing
a constructive solution. Note that the chosen orientation is now the opposite of the orientation chosen in the previous example but the algorithm applies in the exact same manner.

\begin{figure}[!h]
    \centering
    \includegraphics[width=0.8\linewidth]{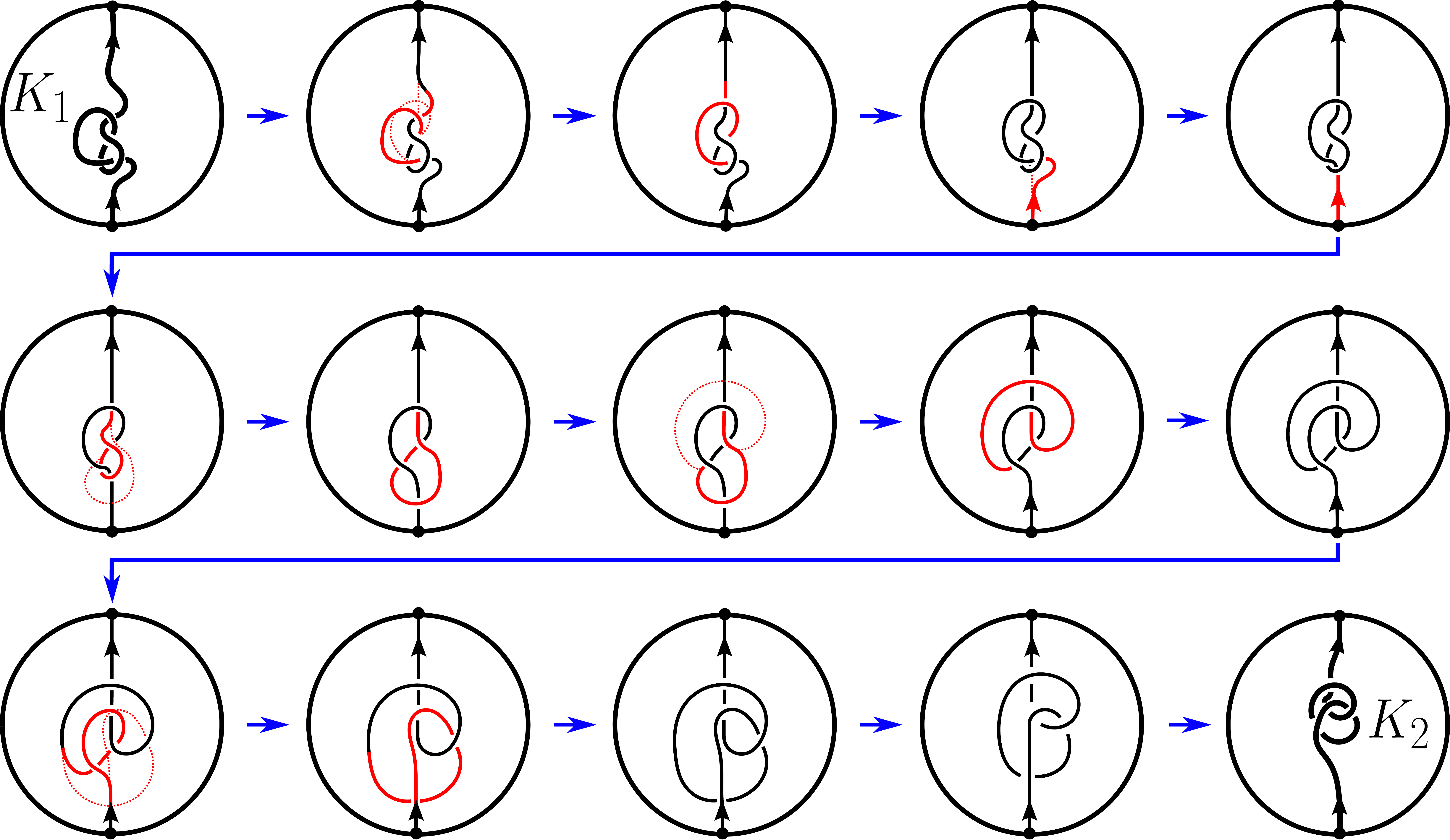}
    \caption{The question of whether the projective knots $K_1$ (top-left) and $K_2$ (bottom-right) are isotopic was raised in  \cite{Narayanan2025}. The figure depicts an explicit isotopy obtained via the algorithm presented in Section \ref{algorithm}, thus yielding a constructive solution to the question.} 
    \label{fig:pair3}
\end{figure}

\subsection{Third application}

In our last application we consider a pair of knots introduced in a preliminary version of the article \cite{Kauffman}, which first appeared as a preprint in arXiv:2401.06050v1 [math.GT] (2024) and which raised the equivalence problem for such pair of knots. These are depicted as $K_1$ (top-left in Figure \ref{fig:pair1}) and $K_2$ (bottom-right in Figure \ref{fig:pair1}). An application of the algorithm from Section \ref{algorithm} yields an explicit isotopy between the two.

\begin{figure}[!h]
    \centering
    \includegraphics[width=0.73\linewidth]{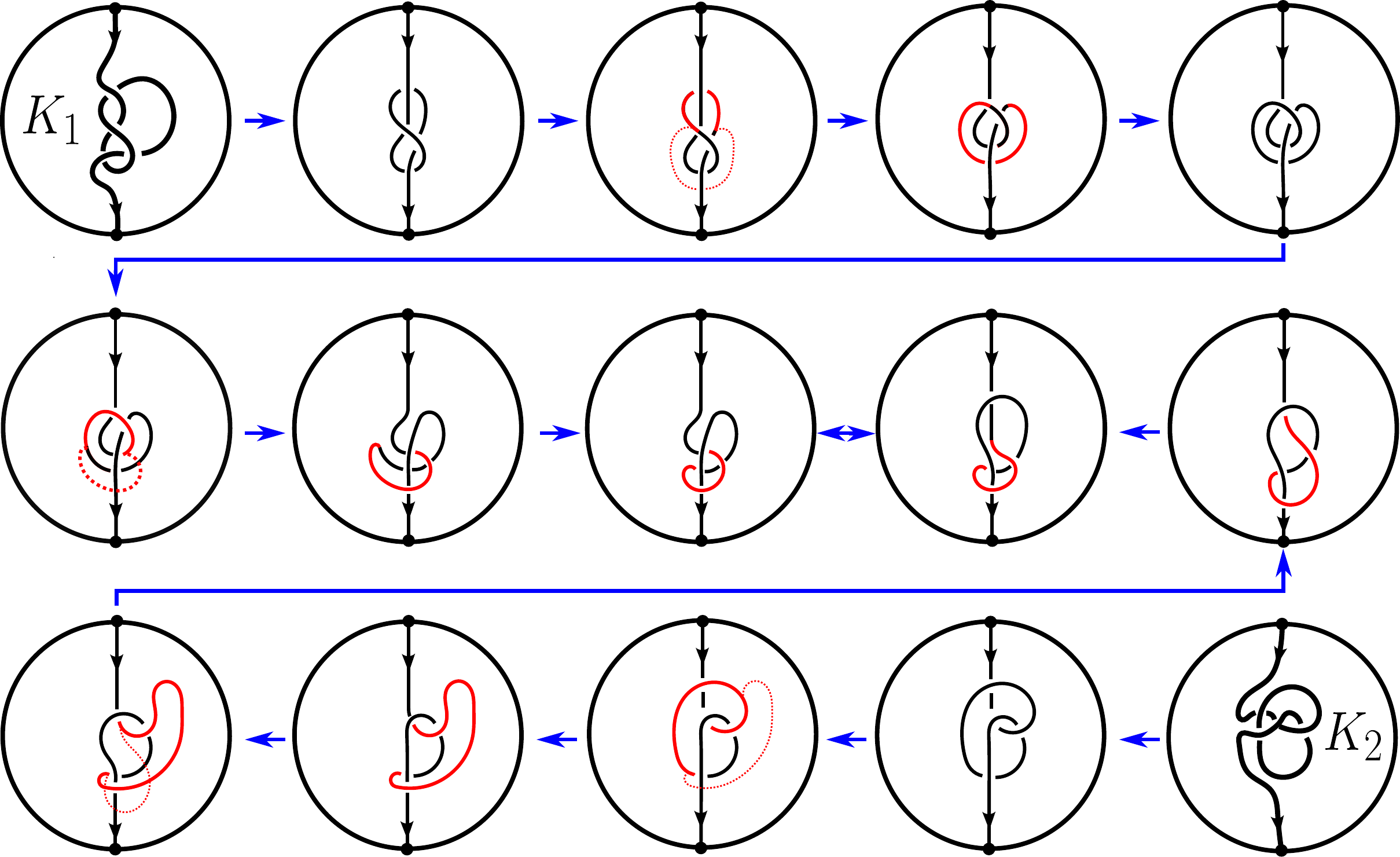}
    \caption{The equivalence problem for the projective knots $K_1$ (top-left) and $K_2$ (bottom-right) was raised in a prepublished version of the work \cite{Kauffman}. The figure exhibits an explicit isotopy between the two obtained via the algorithm presented in Section \ref{algorithm}, therefore yielding a constructive solution to the problem.}
    \label{fig:pair1}
\end{figure}
\bibliographystyle{abbrv}
\bibliography{bibliography}

\end{document}